\documentclass[
  a4paper, 
  reqno, 
  oneside, 
  11pt
]{amsart}

\usepackage[utf8]{inputenc}

\usepackage[dvipsnames]{xcolor} 
\colorlet{cite}{LimeGreen!50!Green}
\usepackage{tikz}  
\usetikzlibrary{arrows,positioning}
\tikzset{ 
  baseline=-2.3pt,
  text height=1.5ex, text depth=0.25ex,
  >=stealth,
  node distance=2cm,
  mid/.style={fill=white,inner sep=2.5pt},
}

\usepackage{amsthm, amssymb, amsfonts}	

\usepackage{graphicx,caption,subcaption}

\usepackage{braket}  
\usepackage{fouriernc}  
\usepackage{microtype}

\usepackage[%
  bookmarks=true,			
  unicode=true,			
  pdftitle={Compactifications of Adjoint Orbits and their Hodge Diamonds},		%
  pdfauthor={Brian Callander}{Elizabeth Gasparim},	%
  pdfkeywords={Hodge diamond}{adjoint orbit}{compactification}{Lefschetz fibration},	
  colorlinks=true,		
  linkcolor=Red,			
  citecolor=cite,		
  filecolor=magenta,		
  urlcolor=RoyalBlue			
]{hyperref}				
\usepackage{cleveref}

\newtheoremstyle{mydef}
  {}		
  {}		
  {}		
  {}		
  {\scshape}	
  {. }		
  { }		
  {\thmname{#1}\thmnumber{ #2}\thmnote{ #3}}	

\theoremstyle{plain}	
\newtheorem{theorem}{Theorem} 
\newtheorem{proposition}[theorem]{Proposition}
\newtheorem{corollary}[theorem]{Corollary}

\theoremstyle{mydef} 

\theoremstyle{remark}

\newtheorem{example}[theorem]{Example}

\newtheorem{remark}{Remark}

\crefname{part}{Part}{Parts}
\crefname{chapter}{Chapter}{Chapters}
\crefname{section}{Section}{Sections}
\crefname{theorem}{Theorem}{Theorems}
\crefname{proposition}{Proposition}{Propositions}
\crefname{lemma}{Lemma}{Lemmata}
\crefname{corollary}{Corollary}{Corollaries}
\crefname{definition}{Definition}{Definitions}
\crefname{example}{Example}{Examples}
\crefname{remark}{Remark}{Remarks}
\crefname{notation}{Notation}{Notations}
\crefname{figure}{Figure}{Figures}
\crefname{enumi}{Item}{Items}

\newcommand{\abs}[1]{\left\lvert #1 \right\rvert}

\newcommand{\homog}{\text{hom}}

\DeclareMathOperator{\ch}{c}

\DeclareMathOperator{\ad}{ad}

\DeclareMathOperator{\Proj}{Proj}

\DeclareMathOperator{\id}{id}

\DeclareMathOperator{\diag}{Diag}

\DeclareMathOperator{\tr}{tr}
\DeclareMathOperator{\Diag}{Diag}

\author{Edoardo  Ballico, Brian  Callander, Elizabeth  Gasparim}
\address{E. Ballico - 
Dept. of Mathematics,
  Univ. of Trento,
38050 Povo (TN), Italy. 
 B. Callander - Imecc -
Unicamp, Depto. de Matem\'{a}tica.  Campinas - SP, Brasil.\qquad
E. Gasparim - Depto. de Matem\'aticas, Universidad Cat\'olica del Norte, Antofagasta, Chile.
E-mails: ballico@science.unitn.it, briancallander@gmail.com, etgasparim@gmail.com.}
\title{Compactifications of Adjoint Orbits and their Hodge Diamonds}
\thanks{E. G. was supported by a Simmons  Associateship Grant from ICTP, Italy. 
}

\begin{document}
\maketitle

\begin{abstract}
A recent theorem of \cite{GGSM1} showed that adjoint orbits of semisimple Lie algebras 
have the structure of symplectic Lefschetz fibrations. 
We investigate the behaviour of  their  fibrewise compactifications.
Expressing   adjoint orbits and  fibres as affine varieties 
in their Lie algebra, we compactify  them to projective 
varieties via  homogenisation of  the defining ideals.
We   find that 
 their Hodge diamonds   vary wildly according to the choice of homogenisation,
and that extensions of the potential to the compactification must acquire degenerate singularities. 
 \end{abstract}

\tableofcontents

\section{Hodge diamonds of Lefschetz fibrations}

Given a symplectic manifold $X$, a symplectic Lefschetz fibration  (SLF) on $X$ is a
surjection $f \colon X \to \mathbb C$ with only Morse type singularities, 
giving $X$ the structure of 
a  locally trivial fibration  
on the complement of the set of critical fibers, and  whose regular  fibres are
symplectic submanifolds of $X$, see  \cite{Se}.  
A large family of new examples of noncompact  SLFs was constructed in the recent paper \cite{GGSM1} and we
need to compactify these examples to obtain  information provided by
their Hodge diamonds (or simply the cohomological dimensions $h^p(\overline X, \Omega^q)$ of the compactification in the singular case).
Our motivation -- coming from mathematical physics -- is to eventually study categories of Lagrangian vanishing cycles.  
These play an essential role in
the Homological Mirror Symmetry conjecture \cite{Ko}, where such a
category appears as the Fukaya category of a Landau--Ginzburg (LG) model (that is, 
 a K\"ahler manifold $X$
equipped with a holomorphic function $f \colon X \to \mathbb C$ called the
superpotential). 
SLFs are nice examples of LG models where a precise definition of the
Fukaya category  of Lagrangian vanishing cycles is available, see \cite{fooo}, \cite{Se}.

  \cite{GGSM1} showed the existence of the structure of 
SLFs on adjoint orbits of semisimple Lie algebras.  These
adjoint orbits are not compact.  In fact, they are diffeomorphic to
cotangent bundles of   flag varieties \cite{GGSM2}.  
We want to compare the behaviour of vanishing cycles on $X$ and on its compactifications.
  Expressing the adjoint orbit as an algebraic variety, we
homogenise its ideal to obtain  a projective variety, which serves as our
compactification.
To study  such a compactification $\overline X$, we calculate its 
cohomological dimensions $h^p(\overline X, \Omega^q)$, as well as those of the compactified fibres of the SLF.
 Calculating such numbers is computationally 
heavy, so we used Macaulay2.
Details of the computational algorithms we used appear in \cite{CG}.
 In the smooth case, these dimensions give us the Hodge diamonds, from which we can read off topological data for the total space
$X$ as well as for the fibres of the SLF.

\begin{remark}
Choosing a compactification is in general a delicate task:  a
different choice of generators for the defining ideal of the orbit can result in completely
different cohomologies  of the corresponding compactification. This happens because the 
homogenisation of an ideal $I$ can change drastically  if we vary the choice
of generators for $I$ (see \cref{reg1}). 
\end{remark}

In \cref{sec:mainTheorem}, we present the principal theorem that furnishes us
with examples.  In \cref{sec:egsl2c}, we find all adjoint orbits of
$\mathfrak{sl}(2,\mathbb C)$ (up to isomorphism), and apply our
compactification process to this simple case. In \cref{sec:2m1m1}, we consider
a more involved example of an adjoint orbit inside $\mathfrak{sl}(3,\mathbb
C)$, corresponding to the minimal flag variety, and show that any extension of 
the potential to the compactified orbit must acquire degenerate singularities, 
hence it would no longer remain a Lefschetz
fibration.  This is
generalised in \cref{sec:generalisation} to the minimal flag variety of
$\mathfrak{sl}(n+1,\mathbb C)$.  We illustrate with an example in
\cref{sec:singular} just how delicate a task compactification can be.
\vspace{3mm}

Acknowledgements:   We are grateful to Koushik Ray and Pushan Majumdar of the Department of Theoretical
  Physics, Indian Association For The Cultivation of Science, Kolkata, for
  running  our large memory (48 GB)  computations on their servers, and to 
  Daniel Grayson for the time  spent
  assisting us with technical issues of Macaulay2. 
  
\section{Lefschetz fibrations on adjoint orbits}
\label{sec:mainTheorem}

Let $H_0$ be an element in the Cartan subalgebra of a semisimple Lie algebra $\mathfrak g$, let 
$\mathcal O(H_0)$ denote its adjoint orbit and $\langle \cdot, \cdot \rangle$ the Cartan-Killing form.
It is proved in \cite{GGSM1} that for each regular element $H \in \mathfrak g$, the function 
$f_H\colon \mathcal O(H_0)\rightarrow \mathbb C $ given by $f_H(x) = \langle
H,x \rangle$ gives the
orbit the structure of a symplectic Lefschetz fibration.
This includes the following properties for $f_H$:
\begin{enumerate}
\item The singularities are 
 nondegenerate.

\item If $c_{1},c_{2}\in \mathbb{C}$ are regular values, then the level
manifolds $f_{H}^{-1}\left( c_{1}\right) $ and $f_{H}^{-1}\left(
c_{2}\right) $ are diffeomorphic.

\item There exists a symplectic form $\Omega $ in $\mathcal{O}\left(
H_{0}\right) $ such that if $c\in \mathbb{C}$ is a regular value then the
level manifold $f_{H}^{-1}\left( c\right) $ is symplectic; that is, the
restriction of $\Omega $ to $f_{H}^{-1}\left( c\right) $ is a symplectic
(nondegenerate) form.

\item If $c\in \mathbb{C}$ is a singular value, then $f_{H}^{-1}\left(
c\right) $ is a union of  affine subspaces (contained in $\mathcal{O}\left(
H_{0}\right) $). These subspaces are symplectic with respect to the form \, $%
\Omega $ from the previous item.
\end{enumerate}

We compactify the orbit by projectivisation; that is, we homogenise the
polynomials with an extra variable $t$
to obtain a projective variety.

\section{Compactification of the orbit of $\mathfrak{sl}(2,\mathbb C)$} 
\label[section]{sec:egsl2c}

Inside $\mathfrak{sl}(2,\mathbb C)$, all adjoint orbits are of the same
isomorphism type, which we now describe as an SLF with $2$ critical values.
In $\mathfrak{sl}(2,\mathbb C)$, take
\[
  H
  = 
  H_0
  =
  \begin{pmatrix}
    1 & 0 \\
    0 & -1
  \end{pmatrix},
\]
which is regular since it has $2$ distinct eigenvalues.  The orbit $\mathcal O
(H_0)$ is the set of matrices in $\mathfrak{sl}(2,\mathbb C)$ with eigenvalues
$1$ and $-1$, which forms a submanifold of complex dimension $2$ of
$\mathfrak{sl}(2, \mathbb C)$. 

The Weyl group $\mathcal W
\simeq S_2$ acts via conjugation by permutation matrices.  The two
singularities are thus $H$ and $-H$.  

We can also express the orbit as an affine variety embedded in $\mathbb C^3$.
Writing a
general element $A \in \mathcal O (H_0)$ as
\[
  A
  =
  \begin{pmatrix}
    x & y \\
    z & -x
  \end{pmatrix},
\]
the characteristic polynomial of $A$ is
\[
  -\left( x - \lambda \right) \left( x + \lambda \right) - yz
  =
  \det \left( A - \lambda \id \right)
  =
  \lambda^2 - 1,
\]
the first equality being derived from explicit calculation and the second
due to the fact that $\tr A = 0$ and $\det A = -1$.  This in turn implies that
the orbit $\mathcal O (H_0) \subset \mathfrak{sl} (2, \mathbb C) \simeq
\mathbb C^3$ is an affine variety $X$ cut out by the equation
\begin{equation}
  \label[equation]{eq:sl2orbit}
  x^2 + yz - 1 = 0.
\end{equation}
We can compactify this variety by homogenising eq.~\ref{eq:sl2orbit} and embedding
$X$ into the corresponding projective variety. This gives the surface cut out by  
$x^2+yz-t^2=0$ in $\mathbb P^3$.
The Hodge
diamond of this compactification is shown in figure \ref{fig:sl2hodge}.
\begin{figure}[ htb ]
  $\begin{array}{ccccc}
    && 1 \cr
    & 0 && 0 \cr
    0 && 2 && 0 \cr
    & 0 && 0 \cr
    && 1 \cr
  \end{array}$
  \caption{The Hodge diamond of the projectivisation of $\mathcal O (
  \Diag(1,-1))$.}
  \label[figure]{fig:sl2hodge}
\end{figure}

The height function is
\[
  f_H (A)
  =
  \tr HA
  =
  \tr
  \begin{pmatrix}
    1 & 0 \\
    0 & -1
  \end{pmatrix}
  \begin{pmatrix}
    x & y \\
    z & -x
  \end{pmatrix}
  =
  2x.
\]
Note that the two critical points belong to distinct fibres.  We can also
express the regular fibre (over zero) as the affine variety in
$\set{(y,z) \in \mathbb C^2}$ cut out by the equation
\[
  yz-1 = 0
\]
since it must satisfy eq.~\ref{eq:sl2orbit} and $x = 0$.  As with the orbit,
we  homogenise this equation and embed the fibre into the corresponding
projective variety cut out by the equations $x=0$ and $yz-t^2=0$ in $ \mathbb P^3$.  This yields the Hodge diamond shown in
fig.~\ref{fig:sl2rfibre}.
\begin{figure}[ hbt ]
  $\begin{array}{ccc}
    & 1 \cr
    0 && 0 \cr
    & 1 \cr
  \end{array}$
  \caption{The Hodge diamond of the projectivisation of the regular fibre over
  zero, where $H=H_0 = \Diag(1,-1)$.}
  \label[figure]{fig:sl2rfibre}
\end{figure}
Note that these compactified fibres have no middle homology.

\section{Smooth compactification of an $ \mathfrak{sl}(3,\mathbb C)$ orbit} 
\label{sec:2m1m1}

The adjoint orbits of $\mathfrak{sl}(3, \mathbb C)$  fall into one of
three isomorphisms types.
Here we present an SLF with 3 critical values. In $\mathfrak{sl} (3, \mathbb C)$, consider the orbit $\mathcal O (H_0)$ of
\[
  H_0 = 
  \begin{pmatrix}
    2 & 0 & 0 \\
    0 & -1 & 0 \\
    0 & 0 & -1
  \end{pmatrix}
\]
under the adjoint action.   We fix the
element
\[
  H = 
  \begin{pmatrix}
    1 & 0 & 0 \\
    0 & -1 & 0 \\
    0 & 0 & 0
  \end{pmatrix}
\]
to define the potential $f_H$.  
%
  A general element
$A \in \mathfrak{sl} \left( 3, \mathbb{C} \right)$ has the form
\begin{equation}\label{gen3}A=\left(\begin{matrix}
             x_1 &  y_1&  y_2\\
             z_1 &  x_2 &  y_3\\
             z_2 &  z_3 &  -x_1 - x_2
  \end{matrix}\right)\text{.}
  \end{equation}
In this example, the adjoint orbit $\mathcal{O} (H_0)$ consists of all the
matrices with the minimal polynomial 
\begin{equation}
  \label{eq:gens}
  (A + \id)(A - 2\id).
\end{equation}
So, 
the orbit is the affine variety cut out by the 
ideal $I$ generated by the polynomial entries of $(A + \id)(A - 2\id)$.
To obtain a projectivisation of $X$, we first homogenise its ideal $I$ with
respect to a new variable $t$, then
take the corresponding projective variety.  
In this case, the  projective variety
$\overline{X}$ is a smooth compactification of $X$. 
We used Macaulay2 \cite{M2} to calculate the Hodge diamonds of
a compactification of the adjoint orbit $\mathcal O (H_0)$, 
obtaining:
\[
  \begin{array}{ccccccccc}
    &&&& 1 &&&& \\
    &&& 0 && 0 &&& \\
    && 0 && 2 && 0 && \\
    & 0 && 0 && 0 && 0 &  \\
    0 && 0 && 3 && 0 && 0 \\
    & 0 && 0 && 0 && 0 & \\
    && 0 && 2 && 0 && \\
    &&& 0 && 0 &&& \\
    &&&& 1 &&&& 
  \end{array} \text{.}
\]

 We now calculate the Hodge diamond of a compactified regular fibre. 
The potential corresponding to our choice of $H$ is 
 $f_H = x_1 - x_2$.
The critical values of this potential are $\pm 3$ and $0$.
Since all regular fibres of an SLF are isomorphic, it suffices to chose the
regular value $1$.  We then define the regular fibre $X_1$  as the
variety in $\mathfrak {sl} (3, \mathbb C) \cong \mathbb C^8$ corresponding to
the ideal $J$ obtained by summing $I$ with the ideal generated by $f_H-1$. 
We then homogenise $J$ to obtain a projectivisation $\overline{X}_1$ of the
regular fibre $X_1$. The Hodge diamond of $\overline{X}_1$ is:
\[
  \begin{array}{ccccccc}
    &&& 1 &&& \\
    && 0 && 0 && \\
    & 0 && 2 && 0 & \\
    0 && 0 && 0 && 0  \\
    & 0 && 2 && 0 & \\
    && 0 && 0 && \\
    &&& 1 &&& 
  \end{array}\text{.}
\]

\begin{remark}
We used the same method to calculate the Hodge diamonds for the singular fibre over $0$ 
and obtained the same Hodge diamond as for the regular fibres.
\end{remark}

\begin{remark}
  More details of this example 
appear 
  in \cite{brianThesis}.
\end{remark}

\section{Generalisations   and computational corollaries}
\label[section]{sec:generalisation}

We generalise our example of $\mathfrak{sl} (3, \mathbb C)$ to
$\mathfrak{sl} (n+1, \mathbb C)$.  To obtain the case where the adjoint orbit is 
diffeomorphic to the cotangent bundle of  the minimal flag, we set
$H_0 = \diag (n, -1, \dotsc, -1)$ and $H = \diag (1,-1,0, \dotsc, 0)$.  Then 
the diffeomorphism type of the adjoint orbit is given by $\mathcal O
(H_0) \simeq
 T^* \mathbb P^{n}$ (see \cite[sec. 2.2]{GGSM2}), and $H$ gives the potential $x_1 - x_2$
as before.  If we  compactify this orbit to $\mathbb P^{n} \times
{\mathbb P^{n}}^{\ast}$ (this may be done holomorphically by \cite[Sec. 4.2]{GGSM2}), then the Hodge classes of the compactification  are
given by 
\begin{equation}\label{diam}h^{p,p} = n+1 - \abs{n-p}
\end{equation}
 and the
remaining  Hodge numbers are $0$.  An application of the Lefschetz hyperplane
theorem determines all but the Hodge numbers of the middle row of the compactification of
the regular fibre, and computations shows the latter are zero.

\begin{remark}\label{bundle} 
  We observe that there are various ways to look at the isomorphism type of the adjoint orbit 
$\mathcal O(H_0)$ depending on the point of view  best suited to a given problem.

 Firstly,
the adjoint orbit $\mathrm{Ad}\left( G\right) H_{0}$ can be identified with the homogeneous
space $G/Z_{H_{0}}$ where $Z_{H_0}$ is the centralizer of $H_0$.  
The compact subgroup $K$ of $G$ cuts out 
 the subadjoint orbit $\mathrm{Ad}\left( K\right) H_{0}$, which can be identified  with the flag manifold $\mathbb{F}_{H_{0}}=G/P_{H_{0}}$
where $P_{H_{0}}$ is the parabolic subgroup associated to $H_{0}$. 
In \cite{GGSM1} the symplectic structure on $\mathrm{Ad}\left( G\right) H_{0}$ is chosen as  the imaginary part 
of the Hermitian form inherited from $\mathfrak g$. With this choice, the flag $\mathbb{F}_{H_{0}}$ is the Lagrangian in $\mathcal O(H_0)\simeq 
T^*\mathbb{F}_{H_{0}}$  corresponding to the zero section of the cotangent bundle. From a Riemannian point of view this is also diffeomorphic to $T\mathbb{F}_{H_{0}}$.

Secondly,  $\mathcal O(H_0)$ can be identified with the open orbit of the diagonal action of $G$ on the product 
$\mathbb{F}_{H_{0}}\times \mathbb{F}_{H_{0}^*}$ \cite[sec. 4.2]{GGSM2}. 
A vector bundle structure on $\mathcal O(H_0)$ is obtained by observing that
\begin{equation*}
\mathcal{O}(H_{0})=\mathrm{Ad}(G)H_0=\mathrm{Ad}(K)(H_0+\mathfrak{n}^+)=\bigcup_{k\in K}\mathrm{Ad}(k)(H_0+\mathfrak{n}^+),
\end{equation*}
where $\mathfrak n^+$ is the sum of the eigenspaces of $\ad (H_0)$ associated to its positive eigenvalues.
The process of projectivisation  then transforms the affine  space $H_0 + \mathfrak n^+$ into a projective space of the same dimension as the flag. 
\end{remark}

\begin{example}
  \label{eg:fibreProj}
  Let $\Sigma = \Set{\alpha_{12}, \alpha_{23}}$ be the usual choice of simple roots for $\mathfrak{sl} (3,\mathbb C)$. 
  In the case $H_0 = \diag(2,-1,-1)$, the corresponding positive nilpotent part is $\mathfrak n^+ = \mathfrak g_{\alpha_{12}} \oplus \mathfrak g_{\alpha_{13}}$, which consists of matrices of the form:
  \[
    \begin{pmatrix}
      0 & y_1 & y_2 \\
      0 & 0 & 0 \\
      0 & 0 & 0 
    \end{pmatrix}.
  \]
  With the description of $\mathcal O (H_0)$ as a vector bundle above, the closure of the fibre $H_0 + \mathfrak n^+$ inside the compactification of section~\ref{sec:2m1m1} consists of matrices of the form:
  \[
    A
    =
    \begin{bmatrix}
      2t & y_1 & y_2 \\
      0  & -t  & 0 \\
      0  & 0   & -t
    \end{bmatrix}.
  \]
  Two matrices of this form are equivalent if one is a scalar multiple of the other and \emph{a priori} one might expect there to be further relations between the matrices.
  However, it can be verified by inspecting the generators of the defining ideal that there are no further relations.
  Therefore, we can embed $H_0 + \mathfrak n^+$ into $\mathbb P^2$ by mapping $A$ to $[t,y_1,y_2]$.
  The case of $\mathfrak{sl}(n, \mathbb C)$ is similar, with a map into $\mathbb P^n$ given by $[t,y_1,\dots, y_n]$.
\end{example}

\begin{proposition} Let $H_0 = \diag (n, -1, \dotsc, -1)$.
Then the adjoint orbit of $H_0$ in  $\mathfrak{sl}(n+1,\mathbb C)$ compactifies holomorphically to a trivial product.
\end{proposition}

\begin{proof}
  For the case  $H_0 = \diag (n, -1, \dotsc, -1)$,
   \cite[Thm.~5.11]{GGSM2} showed that  $\mathcal O(H_0)$
can be embedded holomorphically   into $\overline{X}:=\mathbb P^{n}\times {\mathbb P^{n}}^{\ast}$ as the open orbit of the diagonal action of $G$ on 
$\mathbb{F}_{H_{0}}\times \mathbb{F}_{H_{0}^*} \simeq \mathbb P^{n}\times {\mathbb P^{n}}^{\ast} .$
We claim that the complement of the image of $\mathcal O(H_0)$ in the compactification $\overline{X}= \mathbb P^{n}\times {\mathbb P^{n}}^{\ast}$ 
can be identified with the adjoint orbit of the nilpotent matrix
  $$N=
    \begin{pmatrix}
      0 & 1 & \dots & 0  \\
      0  & 0 &    & 0 \\
      \vdots & & \ddots \\
      0  &    & \dots & 0 
    \end{pmatrix}.
  $$
 Since $G$ acts by conjugation, it is clear that if $N$ belongs to  $\overline{X}$ so does its entire adjoint orbit $\mathcal O(N)$.
 Notice that $\mathcal O(H_0)$ contains all matrices of the form 
  $$N=
    \begin{pmatrix}
      n & t  & \dots & 0  \\
      0  & -1 &    & 0 \\
      \vdots & & \ddots \\
      0  &    & \dots & -1 
    \end{pmatrix},
  $$
this can be verified by direct calculation. Now, dividing by $t$ and taking limit when $t \mapsto \infty$ shows that $N$ belongs to $\overline{X}$
(in fact, this argument  shows that any manifold that serves as a compactification of $\mathcal O(X_0)$  contains a copy of  $N$).
It remains to prove that $\mathbb P^{n}\times {\mathbb P^{n}}^{\ast}\setminus  \mathcal O(X_0)\simeq \mathcal O(N)$.
Consider the diagonal action of $SL(n, \mathbb C)$ in $\mathbb P^{n}\times {\mathbb P^{n}}^{\ast}$ by $g(x,y) = (gx,gy)$.
The orbits of this action on a product of a flag an its dual $\mathbb{F}_{H_{0}}\times \mathbb{F}_{H_{0}^*}$ are in general of the form $G \cdot (b_0, w b_0^*)$ 
where $b_0$ and $b_0^*$ are the origins of  $\mathbb{F}_{H_{0}}$ and  $\mathbb{F}_{H_{0}^*}$ respectively and $w \in \mathcal W$. In the particular case considered here, for 
$H_0= (n, -1, \dots, -1)$,  $\{(b_0,wb_0^*), w \in  \mathcal W\}$ has only 2 elements, and as a consequence, the diagonal action of $G$ on $\mathbb{F}_{H_{0}}\times \mathbb{F}_{H_{0}^*}$ has only 2 orbits, the open one isomorphic to $\mathcal O(H_0)$,  and the closed one isomorphic to $\mathcal O(N)$.
\end{proof}

\begin{remark}
  As mentioned in Remark~\ref{bundle}, under the real diffeomorphism, the flag  $\mathbb F_{H_0}$ corresponds to the zero section of the vector bundle $\mathcal O (H_0)$
 and consequently is Lagrangian in the orbit.  
  This flag remains Lagrangian when embedded into  the product $\mathbb F \times \mathbb F^*$ as the anti-diagonal.
  Therefore, by Weinstein's theorem,  it has  a neighbourhood which is symplectomorphic to the cotangent bundle $T^* \mathbb F$.
  However,  via the equivariant  real diffeomorphism  $\iota \colon \mathcal O(H_0) \rightarrow T^*\mathbb F$
  exhibited in  \cite[Thm.~2.1]{GGSM2}  the canonical symplectic form on the cotangent bundle pulls back to the
  Kostant--Kirillov--Souriau (KKS) form  on the  adjoint orbit $ \mathcal O (H_0)$, thus the real diffeomorphism can not be made holomorphic.
  \end{remark}

The following  corollary follows immediately from observing the Hodge diamonds we obtained.

\begin{corollary} 
An extension of the potential $f_H$ to the compactification  $\mathbb P^n\times {\mathbb P^n}^{\ast}$
cannot be of Morse type; that is, it must have degenerate singularities.
\end{corollary}

\begin{proof} Our potential has singularities at $wH_0, w\in \mathcal W$.  Now
observe that the Hodge diamond of our compactified regular fibres have only zeroes in the middle row, 
hence any extension of the  fibration  to the compactification will have no vanishing cycles. 
However, the existence of a Lefschetz fibration  with singularities and without vanishing cycles  
is precluded by the fundamental theorem of Picard--Lefschetz theory.
\end{proof}

\section{Singular compactifications of  $\mathfrak{sl}(3, \mathbb C)$ orbits}
\label{sec:singular}

We show that the  compactified regular fibre for $f_H$ can change drastically 
according to the choice of homogenisation of the ideal cutting out the orbit as an affine variety. 
The compactifications obtained in this section turn out to be singular. Nevertheless, we wish to
depict diamonds with their sheaf cohomological information. It is well known, see e.g. \cite{105} 
that every complex  algebraic variety has a mixed Hodge structure. We do not attempt to 
describe  mixed Hodge  structures, instead we calculate the numbers $h^p(\overline{X}, \Omega^q)$,
where $\Omega$ is the cotangent sheaf. 
Although  we do not explore here how the diamond containing such numbers might be related to 
the topology of $\overline{X}$, such diamonds do provide us with enough information to show 
that 2 natural choices of compactification differ.

\subsection{A fibration with 4 critical values}
\label{reg1}
In $\mathfrak{sl}\left( 3,\mathbb{C}\right) $ we take 
\[
H=H_{0}=\left( 
\begin{array}{ccc}
1 & 0 & 0 \\
0 & -1 & 0 \\ 
0 & 0 &  0%
\end{array}%
\right) ,
\]%
which  is regular since it has 3 distinct eigenvalues.
Then $X= \mathcal{O}\left( H_{0}\right) $ is the set of 
matrices in  $\mathfrak{sl}\left( 3,\mathbb{C}\right) $ with eigenvalues $1,0,-1$. 
This set forms a submanifold of real dimension $6$ (a complex threefold).

In this  case $\mathcal{W}\simeq S_3$, the permutation group in 3 elements, and acts via
conjugation by permutation matrices. Therefore, the potential 
$f_H=x_1-x_2$ has 6 singularities;
namely, the 6 diagonal matrices with diagonal entries $1,0,-1$.
The four singular values of $f_H$ are $\pm 1, \pm2$.
  Thus, $0$ is a regular value for $f_H$.
Let 
$A \in \mathfrak{sl}(3,\mathbb C)$ be a general element written as in (\ref{gen3}), 
and let $p= \det(A)$, $q=\det(A-\id)$. The ideals $\langle p,q\rangle $ and $\langle p-q,q\rangle $ are clearly 
identical and either of them defines the orbit though $H_0$ as an affine variety in $ \mathfrak{sl}\left( 3,\mathbb{C}\right) $.
Now
$$I = \langle p,q,f_H\rangle \qquad J=\langle p,p-q,f_H\rangle $$ 
are two  identical ideals cutting out the regular fibre $X_0$ over $0$. Let
$I_{\homog}$ and $J_{\homog}$ be the respective saturated  homogenisations and 
notice that $I_{\hom}\neq J_{\hom}$, so that  they define distinct projective
varieties, and thus two distinct compactifications 
\begin{align*}
  \overline X_0^I &= \Proj (\mathbb C[ x_1,x_2,y_1,y_2,y_3,z_1,z_2,z_3,t]/I_{\hom}) 
  \quad \text{and}
  \\
  \overline X_0^J &= \Proj(\mathbb C[x_1,x_2,y_1,y_2,y_3,z_1,z_2,z_3,t]/J_{\hom}) 
\end{align*}
of $X_0$.
Their  diamonds  are given in figure \ref{fig:110rcomparison}.
Remark \ref{?} explains the computational issues. 

\begin{remark}
  The variety $\overline X_0^J$ is an irreducible component of $\overline X_0^I$.
  Indeed, we find that $I \subset J$ and that $J$ is a prime ideal (whereas $I$  is not), thus the variety $\overline X_0^J$ is irreducible and contained in $\overline X_0^I$.
  
\end{remark}

\begin{figure}[htp]
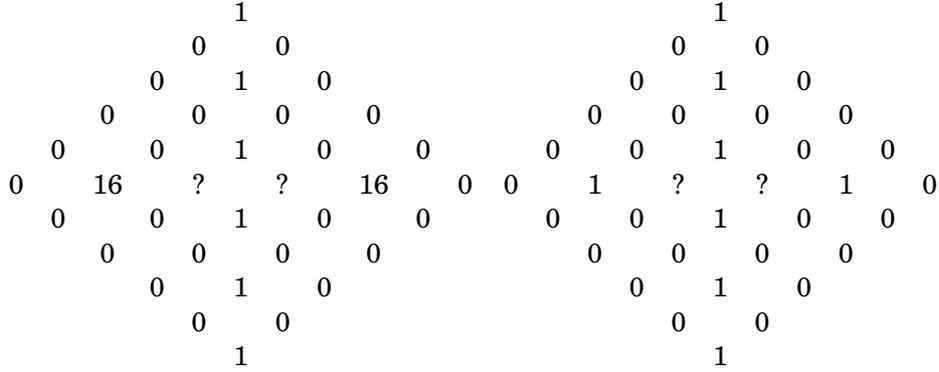

  \begin{subfigure}[b]{0.45\textwidth}
    $\begin{array}{ccccccccccc}
      &&&&& 1 \cr
      &&&& 0 && 0 \cr
      &&& 0 && 1 && 0  \cr
      && 0 && 0 && 0 && 0  \cr
      & 0 && 0 && 1 && 0 && 0  \cr
       0 && 16 && ? && ? && 16 && 0 \cr
      & 0 && 0 && 1 && 0 && 0  \cr
      && 0 && 0 && 0 && 0  \cr
      &&& 0 && 1 && 0  \cr
      &&&& 0 && 0  \cr
      &&&&& 1 \cr
    \end{array}$
    \label{fig:110ur}
  \end{subfigure}
  \qquad
  \begin{subfigure}[b]{0.45\textwidth}
    $\begin{array}{ccccccccccc}
      &&&&& 1 &&&&&    \\
      &&&& 0 && 0 &&&&    \\
      &&& 0 && 1 && 0 &&&    \\
      && 0 && 0 && 0 && 0 &&    \\
      & 0 && 0 && 1 && 0 && 0 &    \\
       0 & & 1 && ? && ? && 1 && 0    \\
      & 0 && 0 && 1 && 0 && 0 &    \\
      && 0 && 0 && 0 && 0 &&    \\
      &&& 0 && 1 && 0 &&&    \\
      &&&& 0 && 0 &&&&    \\
      &&&&& 1 &&&&&    
    \end{array}$
    \label{fig:110sr}
  \end{subfigure}
  \caption{The  diamonds of two projectivisations $\overline X_0^I$ (left) and $\overline X_0^J$ (right) of  the regular fibre
  corresponding to $H=H_0=\Diag(1,-1,0)$.}
  \label{fig:110rcomparison}
\end{figure}

\begin{remark} 
  [Computational matters]\label{?}
  Macaulay2 greatly facilitates cohomological calculations that
  are unfeasible by hand. The Macaulay2 algorithm that computes 
  $h^p(X, \Omega_X^q)$ is written for a smooth variety $X$. However, the algorithm proceeds by 
  resolving the cotangent sheaf and calculating its exterior powers to compute sheaf cohomology, all of which works 
  out reasonably  well  for our singular examples.
  The only drawback is that the memory requirements rise steeply with the
  dimension of the variety -- especially for the  classes $h^{p,p}$.  
  In fact, the unknown entries in our  diamonds (marked with a `?')
  exhausted the
  48GB of RAM of the computers of our collaborators at IACS 
   without producing an answer.
\end{remark}

\subsubsection{Expected Euler characteristic}
To reassure ourselves about the much larger values occurring  for the diamond
of  $\overline X_0^I$ in comparison to  $\overline X_0^J$, we perform the rather amusing calculation of the 
expected Euler characteristic of both varieties, which  give out quite surprising numbers.

\begin{remark} Let $Y = Y_1\cap \dots \cap Y_r$ be a complete intersection. If $Y$ is smooth, then 
the Euler characteristic of $Y$ is uniquely determined by its cohomology class. However, for a singular variety this is 
no longer true, and the cohomological classes $Y_i$ do not determine the topological Euler characteristic. They determine only 
what is called the expected Euler characteristic of $Y$ (equal to the Fulton--Johnson class), see \cite{Cy}. 
\end{remark}

To calculate the expected Euler characteristic we use the following basic formulae from intersection theory.
 Let $X := \mathcal V (f_1, \dotsc, f_k) \subset \mathbb P^{n+k}$ be a complete 
intersection with inclusion
   $i \colon X \to \mathbb P^{n+k}$.
  Define $\alpha := i^* \left( \ch_1 \left( \mathcal O_{\mathbb
    P^{n+k}} (1) \right) \right) \in H^2 \left( X \right)$.  Then
  \begin{align}
    \int_X \alpha^{\wedge n}
    =
    d,
    \label{eq:d}
  \end{align}
  where $d = \prod_1^k d_i$ and  $d_i = \deg f_i$.
 Moreover,
  \begin{align}
    \ch \left( X \right)
    =
    \frac{(1+\alpha)^{n+k+1}}{\prod_i (1+d_i \alpha)}
    =
    1 + \ch_1 \left( X \right) + \dotsb + \ch_n \left( X \right),
    \label{ch}
  \end{align}
  and the Euler characteristic is given by
  \begin{align}
    \chi \left( X \right)
    =
    \int_X \ch_n \left( X \right),
    \label{eq:chi}
  \end{align}
  where $\ch_i \left( X \right) \in H^{2i} \left( X \right)$ is the $i$-th
 Chern class.

\begin{example}
We first illustrate the formula with two elementary cases.

For a conic $C$  in $\mathbb P^2$, expression \ref{ch} 
produces  $(1+\alpha)^3/(1+2\alpha)$, whose expansion at zero is 
$1+\alpha +\alpha^2+o(\alpha^3)$.
Here, $\int \alpha = 2$ 
and we get $\chi(C) =2$, which  was to be expected since the conic is topologically
isomorphic to $\mathbb P^1$.

For the quartic $Q$ in $\mathbb P^3$, expression \ref{ch} gives
 $ (1+\alpha)^4/(1+4\alpha)$, whose 
expansion at zero is
$1+6\alpha^2+o(\alpha^3)$. 
Here, $\int \alpha^2 = 4$ 
and so $\chi(Q)= 6 \times 4 = 24$,
which was to be expected since the quartic is a $K3$ surface,
whose Hodge diamond is well known to be 
\begin{displaymath}
 \begin{array}{ccccc}
 & & 1 & & \\
 & 0 & & 0 & \\
 1 & & 20 & & 1 \\
 & 0 & & 0 & \\
 & & 1 & &
 \end{array}\text{.}
 \end{displaymath}
\end{example}

Now let us return to our  two projectivisations $\overline X_0^I$  and $\overline X_0^J$.
For the ideal $I_{\hom}$ we have degrees $d_1=d_2=3$ and $d_3=1$.
The orbit was embedded in $\mathbb P^8$.
So  expression \ref{ch} gives
$$c\left(\overline X_0^I \right) = \frac{(1 + \alpha )^9}{(1+3 \alpha)(1+3\alpha)(1+\alpha)}= 
 \frac{(1+\alpha)^8 }{(1+3 \alpha)^2}.
$$
The Taylor series expansion around zero is given by
 $1+2\alpha+7\alpha^2-4\alpha^3+31\alpha^4-94\alpha^5+o(\alpha^6)$.
Here $\int \alpha^5 = 9$ and we get  the expected Euler characteristic to be
$$\chi \left( \overline X_0^I \right) = -94 \times 9 = -846.$$

On the other hand, for the ideal $J_{\hom}$ we have degrees $d_1=2$, $d_2=3$,
and $d_3=1$.
Expression \ref{ch} gives 
$$c\left(\overline X_0^J\right) = \frac{(1 + \alpha )^9}{(1+2 \alpha)(1+3\alpha)(1+\alpha)}= \frac{ (1+\alpha)^8}{(1+3\alpha)(1+2\alpha)}.$$
The Taylor series expansion around zero is 
$1+3\alpha+7\alpha^2+3\alpha^3+13\alpha^4-27\alpha^5+o(\alpha^6)$.
In this case, $\int \alpha^5 = 6$ and we obtain
$$\chi \left( \overline X_0^J \right) = -27 \times 6 = -162.$$
The difference between $\chi \left( \overline X_0^J \right)$ and $\chi \left(
\overline X_0^I \right)$ is a concrete topological difference between
our two compactifications.

%
%
%
%
%
%

\subsection{A fibration with 6 critical values}
In $\mathfrak{sl}\left( 3,\mathbb{C}\right) $ we now take 
\[
H_{0}=\left( 
\begin{array}{ccc}
3 & 0 & 0 \\
0 & -1 & 0 \\ 
0 & 0 &  -2%
\end{array}%
\right) ,
\]%
which is regular since it has 3 distinct eigenvalues.
Then $\mathcal{O}\left( H_{0}\right) $ is the set of 
matrices in  $\mathfrak{sl}\left( 3,\mathbb{C}\right) $ with eigenvalues $3,-1,-2$. 
We  choose 
\[
H=\left( 
\begin{array}{ccc}
1 & 0 & 0 \\
0 & -1 & 0 \\ 
0 & 0 &  0%
\end{array}%
\right) ,
\]%
giving the potential $f_H(A) = x_1-x_2, $  with critical values $\pm 1, \pm 4, \pm 5$.
This fibration is only mildly different from the previous one by the fact that 2 singular fibres contain 2 singularities each. 
The orbit is diffeomorphic to the 
one of subsection \ref{reg1}. The regular fibres are pairwise diffeomorphic.

As in \ref{reg1}, let 
$A \in \mathfrak{sl}(3,\mathbb C)$, 
and $p= \det(A+ \id)$, $q=\det(A+2\id)$. Once again, the ideals $\langle p,q\rangle $ and $\langle p-q,q\rangle $ are clearly 
equal and either of them defines the orbit though $H_0$ as an affine variety in $ \mathfrak{sl}\left( 3,\mathbb{C}\right) $.
The matrix $A$  belongs to the regular  fibre $X_0$ if in addition it satisfies 
$f_H=x_1-x_2=0$. 
Now, let 
$$I = \langle p,q,f_H\rangle \qquad J=\langle p,p-q,f_H\rangle $$ 
be two  equal ideals cutting out the regular fibre $X_0$ through $0$ and let $I_{\hom}$ and $J_{\hom}$ be the respective homogenisations. 
However, $I_{\hom}\neq J_{\hom}$, so they define distinct projective varieties.
Performing the necessary computations, we obtain the same cohomological diamonds, and the  same Euler characteristics as for the corresponding 
varieties of \ref{reg1}.
\vspace{3mm}

We then went further to check for the appearances of $16$'s and $1$'s in the  diamonds of 
the singular fibres at $1$ and indeed, they reappeared.
\begin{figure}[htp]
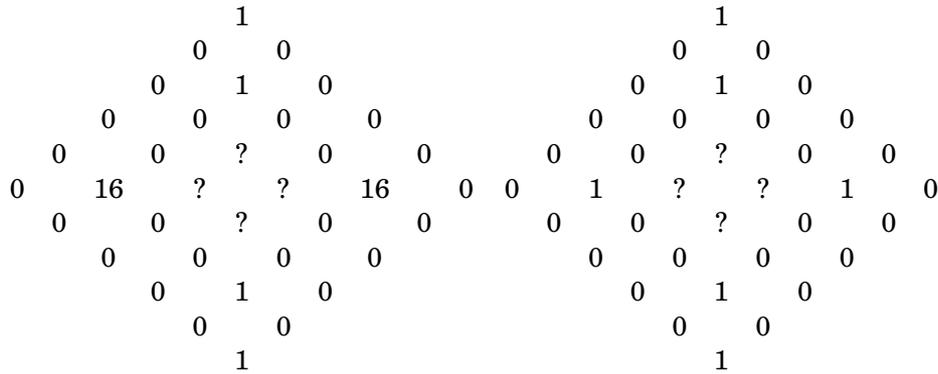

  \begin{subfigure}[b]{0.45\textwidth}
    $\begin{array}{ccccccccccc}
      &&&&& 1    \cr
      &&&& 0 && 0    \cr
      &&& 0 && 1 && 0    \cr
      && 0 && 0 && 0 && 0    \cr
      & 0 && 0 && ? && 0 && 0    \cr
       0 && 16 && ? && ? && 16 && 0   \cr
      & 0 && 0 && ? && 0 && 0    \cr
      && 0 && 0 && 0 && 0   \cr
      &&& 0 && 1 && 0    \cr
      &&&& 0 && 0    \cr
      &&&&& 1    \cr
    \end{array}$
    \label{fig:321us}
  \end{subfigure}
  \qquad
  \begin{subfigure}[b]{0.45\textwidth}
    $\begin{array}{ccccccccccc}
      &&&&& 1 &&&&&    \\
      &&&& 0 && 0 &&&&    \\
      &&& 0 && 1 && 0 &&&    \\
      && 0 && 0 && 0 && 0 &&    \\
      & 0 && 0 && ? && 0 && 0 &    \\
       0 && 1 && ? && ? && 1 && 0    \\
      & 0 && 0 && ? && 0 && 0 &    \\
      && 0 && 0 && 0 && 0 &&    \\
      &&& 0 && 1 && 0 &&&    \\
      &&&& 0 && 0 &&&&    \\
      &&&&& 1 &&&&&    
    \end{array}$
    \label{fig:321ss}
  \end{subfigure}
  \caption{The  diamonds of two projectivisations of the singular fibre
  over $1$  corresponding to $H_0=\Diag(3,-2,-1)$, $H=\Diag(1,-1,0)$.}
  \label{fig:321scomparison}
\end{figure}

 \begin{remark} While we were making the  amendments  to an earlier version of this work.
 Katzarkov, Kontsevich, and Pantev posted \cite{KKP}, which gives 3 definitions of Hodge 
 numbers for Landau--Ginzburg models.  Understanding the relation between the diamonds we gave here and 
 those  Hodge numbers now provides and entirely new perspective for our work. 
 \end{remark}

 \section{Open questions}
 \label{sec:open}
We finish by posing  the following open questions. How many 
compactifications can be obtained via homogenisation?
Is there a preferred choice in the sense that it maintains the topology closest to the original 
variety?
Given two compactifications with distinct numerical invariants, do there exist 
compactifications realising the intermediate values of the invariants?

\end{document}